\documentclass[reqno,12pt,centertags]{article}
\usepackage{amsmath,amsthm,amscd,amssymb,latexsym,upref,stmaryrd}
\usepackage[cp1251]{inputenc}
\usepackage[english]{babel}
\usepackage{mathtext}
\usepackage{mathrsfs}
\usepackage{graphicx}
\usepackage[numbers,sort&compress]{natbib}
\usepackage{amsmath}

\numberwithin{equation}{section}
\numberwithin{figure}{section}
\usepackage{hyperref}
\newcommand*{\mailto}[1]{\href{mailto:#1}{\nolinkurl{#1}}}
\usepackage{pgfplots}
\usepackage{amsmath}
\usepackage{amsthm}

\newtheorem{theorem}{Theorem}
\newtheorem{lemma}{Lemma}

\theoremstyle{proposition}
\newtheorem{proposition}{Proposition}
\newtheorem{ip}{Inverse Problem}
\newtheorem{algorithm}{Algorithm}
\begin{document}

\thispagestyle{empty}

\noindent{\large\bf An inverse  problem for the  matrix Schr\"{o}dinger operator on the half-line with a general boundary condition}
\\

\noindent {\bf Xiao-Chuan Xu}\footnote{School of Mathematics and Statistics, Nanjing University of Information Science and Technology, Nanjing, 210044, Jiangsu,
People's Republic of China, {\it Email:
xcxu@nuist.edu.cn}}
\noindent {\bf and Yi-Jun Pan}\footnote{School of Mathematics and Statistics, Nanjing University of Information Science and Technology, Nanjing, 210044, Jiangsu,
People's Republic of China, {\it Email: yj00262024@163.com
}}
\\

\noindent{\bf Abstract.}
{In this work, we study the inverse spectral problem, using the Weyl matrix as the input data, for the  matrix Schr\"{o}dinger operator on the half-line with the boundary condition being the form of the most general self-adjoint. We prove the uniqueness theorem, and derive the main equation and prove its solvability, which yields a theoretical reconstruction algorithm of the inverse problem.}



\section{Introduction}
Consider the matrix Schr\"{o}dinger equation on the half-line
  \begin{equation}\label{k1}
\ell Y:=-Y''+Q(x)Y=\lambda Y ,\quad x>0,
  \end{equation}
with the boundary condition
\begin{equation}\label{k2}
 T(Y):= A( Y'(0) -hY (0)) - A^\perp Y (0) = 0_n,
\end{equation}
where $Q(x)=(Q_{ls})_{n\times n}$ is an $n\times n$ matrix-valued function in $L^1(0,\infty)$,
namely,
\begin{equation*}
\int_0^\infty \|Q(x)\|dx<\infty,\quad \|Q(x)\|:=\max_{l}\sum_{s=1}^n|Q_{ls}(x)|,
\end{equation*}
the function $Y$ is  column vector-valued or  matrix-valued, $A$ is an orthogonal projection matrix (i.e., $A^\dag =A=A^2$), $ A^\perp=I-A$, $h\in \mathbb{C}^{n\times n}$  satisfies $h=AhA=Ah=hA$, and $0_n$ denotes the zero vector or zero matrix. Denote \eqref{k1} and \eqref{k2} by the problem $L(Q,A,h)$. If $Q(x)^\dag=Q(x)$ and $h^\dag=h$, then the problem $L(Q,A,h)$ is self-adjoint and \eqref{k2} is  the most general self-adjoint boundary condition.

The matrix Schr\"{o}dinger equation
 with the most general self-adjoint boundary condition  appear in quantum mechanics, electronics, nanoscience and other branches of science and engineering and have important applications in quantum mechanics involving particles of internal structures such as spins, scattering on graphs and quantum wires, as well as the matrix-valued integrable system (see \cite{BCFK,SBF,MH1,NW,KS,KS1,KS2,EKKST,PK,PK1,ZS} and the references therein). In particular, various connection conditions  at a vertex of quantum graph can be expressed in terms of the most general  self-adjoint boundary condition (see, e.g., \cite{PK}). For example, the $\delta$-type boundary condition
\begin{equation*}
  y_1(0)=\cdot \cdot\cdot=y_n(0),\quad \sum_{j=1}^n y_j'(0)=ay_1(0)
\end{equation*}
can be expressed by \eqref{k2} with $h=aA$ and $A=(a_{ls})_{n\times n}$, $a_{ls}=1/n$, $l,s=1,...n$ (see \cite{NB0}).

For the  matrix Schr\"{o}dinger equation \eqref{k1} on the half line,  the direct and inverse scattering problems were studied in \cite{AM,AWD,AKW,AKW1,AW,AW1,MH3,XY,XB} and other works, where 
the potential matrix is self-adjoint and satisfies  $(1+x)Q(x)\in L^1(0,\infty)$.
The direct and inverse spectral problems on the half line, using the Weyl matrix as the input data, were studied in \cite{FY,NB2}, where $Q\in L^1(0,\infty)$ but the boundary condition is the Robin boundary condition
\begin{equation}\label{xj1}
    Y'(0)-hY(0)=0_n.
\end{equation}

 Motivated by these mentioned works of the  matrix Schr\"{o}dinger operator on the half line, we are interested in studying the inverse spectral problem, using the Weyl matrix as the input data, for  the matrix Schr\"{o}dinger equation \eqref{k1} with the most general self-adjoint boundary condition \eqref{k2}, namely, the problem $L(Q,A,h)$. 
In this paper, we consider a more general case, that is, we  don't require $h=h^\dagger$ or $Q(x)^\dag=Q(x)$. 
 When $A=I_n$, it becomes the problem considered in \cite{FY}. When $A$ is a zero matrix, then  \eqref{k2} is the Dirichlet boundary condition.
  Based on partial ideas from \cite{FY} (the method of spectral mapping), we overcome the difficulties caused by the general boundary condition \eqref{k2}, and prove that the Weyl matrix uniquely recovers the matrices $A,h$ and the potential $Q(x)$. Moveover, after deducing some important estimates, we establish the main equation and show its solvability in a constructed Banach space, which yields a a theoretical reconstruction
algorithm for the inverse problem.

Let us also mention that some aspects of the matrix Schr\"{o}dinger operator on the full line were investigated in \cite{AKW0,NB,Dv}, and  the matrix Schr\"{o}dinger operators on a finite interval were studied in  \cite{RC,CK,CK1,CY,VY,NB0,NB3,NB4,X0,Sh} and other works, in particular,  the most self-adjoint boundary condition were considered in \cite{NB0,NB3,NB4,X0}.
It should be pointed out that in addition to \eqref{k2}, there are some other types of the most self-adjoint boundary conditions, such as
 \begin{equation}\label{xc1}
    A_1^\dagger Y'(0)-B_1^\dagger Y(0)=0_n,
\end{equation}
where $A_1$ and $B_1$ satisfy one of the following three conditions
\begin{equation}\label{xc3}
B_1^\dag A_1=A_1^\dag B_1, \quad A_1^\dagger A_1 +B_1^\dagger B_1>0,
\end{equation}
\begin{equation}\label{oo1}
 B_1^\dag A_1=A_1^\dag B_1, \quad {\rm rank}\left[-B_1^\dag, A_1^\dag\right]=n,
\end{equation}
 \begin{equation}\label{xc2}
 A_1=\frac{1}{2} (U+I_{n}), B_1=\frac{i}{2} (U-I_{n}),
 \end{equation}
 where $U$ is a unitary matrix. The condition \eqref{xc3} was considered by Aktosun \emph{et al} \cite{AWD,AKW,AKW1,AW,AW1}. The condition \eqref{oo1} appears in \cite{KS}, and the condition \eqref{xc2} was studied in \cite{MH3,XY,XB}. It was proved in \cite{AKW}  that \eqref{xc3}, \eqref{oo1} and \eqref{xc2} are equivalent to each other. The  most self-adjoint boundary condition of the form \eqref{k2} was recently studied by Bondarenko \cite{NB0,NB3,NB4} in studying the matrix Schr\"{o}dinger operator on a finite interval. In Appendix, we analyse how to transform the condition \eqref{xc1} with \eqref{xc2} into the form of \eqref{k2}.

The paper is organized as follows. In Section 2, we introduce four matrix solutions of the matrix  Schr\"{o}dinger equation, and define the Weyl matrix $M(\lambda)$ and study its properties. In Section 3, we prove a uniqueness theorem for the inverse problem.  In Section 4, we establish the main equation and show its solvability, which yields a theoretical algorithm for
the inverse problem. 

\section{ The Weyl matrix}
In this section, we introduce the Weyl matrix and give its properties. Let $\lambda=\rho^2$ and $\Omega:=
 \left\{\rho:{\rm Im}\rho\ge 0 , \rho \ne 0 \right\}.$ It is known that the equation \eqref{k1}
  has the Jost solution $e(x,\rho)$, which is a matrix-valued function satisfying the integral equation
\begin{equation}\label{jost}
  e(x,\rho)=e^{i\rho x}I_n+\int_x^\infty\frac{\sin \rho(t-x)}{\rho}Q(t)e(t,\rho)dt,\quad x\ge0,  \quad \rho\in \Omega.
\end{equation}
Using the method of successive approximation for \eqref{jost}, we obtain that for each fixed $x\ge 0$, the matrix-functions $e^{(\nu)}(x,\rho)$ ($\nu =0,1$) are analytic for ${\rm Im}\rho >0$, and are continuous for $\rho \in \Omega$, and have the asymptotics (see, e.g., \cite{AW})
\begin{equation}\label{jost1}
  e(x,\rho)=e^{i\rho x}\left[I_n+\frac{1}{i\rho }[-\omega(x,0)+\omega(x,\rho)]+O\left(\frac{1}{\rho^2}\right)\right],  \quad |\rho|\to\infty ,\quad \rho\in \Omega,
\end{equation}
\begin{equation}\label{jost2}
  e'(x,\rho)=i\rho e^{i\rho x}\left[I_n-\frac{1}{i\rho}(\omega(x,0)+\omega(x,\rho))+O\left(\frac{1}{\rho^2}\right)\right],  \quad |\rho|\to\infty ,\quad \rho\in \Omega ,
\end{equation}
where
\begin{equation*}
  \omega(x,\rho)=\frac{1}{2}\int_x^\infty Q(t)e^{2i\rho (t-x)}dt.
\end{equation*}

Define the Jost matrix
\begin{equation}\label{yj26}
    J(\rho):=T(e(x,\rho))=A(e'(0,\rho)-he(0,\rho))-A^\perp e(0,\rho).
\end{equation}
\begin{lemma}\label{lemA2}
The Jost matrix $J(\rho)$ defined in \eqref{yj26} is analytic for ${\rm Im}\rho >0$, and continuous for $\rho \in \Omega$, and satisfies the asymptotics
\begin{equation}\label{k39}
    J(\rho)=J_{0}(\rho)\left[I_{n}-\frac{h+\omega(0,0)}{i\rho}+\frac{\kappa(\rho)}{i\rho}+O\left(\frac{1}{\rho^2}\right)\right],\quad \lvert\rho\rvert\to\infty,\quad \rho\in\Omega,
\end{equation}
where
\begin{equation}\label{k40}
J_{0}(\rho):=i\rho A-A^{\perp},\quad  \kappa(\rho):={(A^\perp-A)\omega(0,\rho)}=\frac{A^\perp-A}{2}\int_0^\infty Q(t)e^{2i\rho t}dt.
\end{equation}
\end{lemma}
\begin{proof}
 Using \eqref{jost1} and \eqref{jost2} in \eqref{yj26}, we have that as $\lvert\rho\rvert\to\infty $ in $\Omega$,
 \begin{align}\label{jost3}
\notag J(\rho) = &i\rho A -A^{\perp}-A\left[h+\omega(0,0)+\omega(0,\rho)+O\left(\frac{1}{\rho}\right) \right]+\frac{A^\perp}{i\rho}\left[\omega(0,0)-\omega(0,\rho)+O\left(\frac{1}{\rho}\right)\right]\\
\notag =&J_0(\rho)\left[I_n-J_0(\rho)^{-1}A\left(h+\omega(0,0)+\omega(0,\rho)+O\left(\frac{1}{\rho}\right) \right)\right.\\
&\left.\qquad\qquad+J_0(\rho)^{-1}\frac{A^\perp}{i\rho}\left(\omega(0,0)-\omega(0,\rho)+O\left(\frac{1}{\rho}\right)\right)\right].
 \end{align}
Since $A$ is an orthogonal projection matrix, it is easy to get that
\begin{equation}\label{xj3}
(i\rho A-A^\perp )^{-1}A=\frac{A}{i\rho},\quad (i\rho A-A^\perp )^{-1}A^\perp =-A^\perp.
\end{equation}
Using \eqref{xj3} in \eqref{jost3}, and noting $A+A^\perp =I_n$ and $Ah=h$, we get \eqref{k39}.
 The proof is complete.
\end{proof}

Denote by $\Pi$ the $\lambda$-plane with the cut $\lambda\ge0$, and $\Pi_1=\overline{\Pi}\setminus\{0\}$. Note that $\Pi_1$ and ${\Pi}$ must be considered as subsets of the Riemann surface of the square root function. Then, under
the map $\rho\to\rho^2=\lambda$, the domain $\Omega$ corresponds to $\Pi_1$.
Introduce the \emph{Weyl solution}
\begin{equation}
   \phi(x,\lambda) :=e(x,\rho)J(\rho)^{-1},\quad \rho \in \Omega.
\end{equation}
It obvious that the Weyl solution $\phi(x,\lambda)$ satisfies
\begin{equation}\label{xc7}
    T(\phi)=I_{n}.
\end{equation}
Due to \eqref{jost1}, \eqref{jost2} and \eqref{k39}, we have that as $\lvert\rho\rvert\to\infty,\rho\in\Omega,\nu=0,1$
\begin{equation}\label{k42}
    \phi^{(\nu)}(x,\lambda)=(i\rho)^{\nu}e^{i\rho x}(I_{n}+O(1/\rho))(i\rho A-A^{\perp} )^{-1}.
\end{equation}
Define \begin{equation}\label{k2.12}
 M(\lambda):=A\phi (0,\lambda)+A^{\perp}\phi' (0,\lambda)=\left[Ae (0,\rho)+A^{\perp}e' (0,\rho)\right]J(\rho)^{-1},
\end{equation}
which is called the \emph{Weyl matrix} for the problem $L(Q,A,h)$.
Denote
\begin{equation}
    \Lambda=\{\lambda=\rho^{2}:\rho \in\Omega, \det J(\rho)=0\},
\end{equation}
\begin{equation}
    \Lambda'=\{\lambda=\rho^{2}:{\rm Im}\rho >0, \det J(\rho)=0\},
\end{equation}
\begin{equation}
    \Lambda''=\{\lambda=\rho^{2}:{\rm Im}\rho=0, \rho \neq 0,  \det J(\rho)=0\}.
\end{equation}
Clearly, $\Lambda =  \Lambda' \cup \Lambda ''$ is a bounded set and $\Lambda'$ is a bounded and at most countable set. By the definition of the Weyl matrix, we can easily obtain the  analytic property of  $M(\lambda)$. Let us summarize it in the following proposition.
\begin{proposition}
 The Weyl matrix $M(\lambda)$ is analytic in $\Pi\setminus\Lambda'$ and continuous in $\Pi_1\setminus \Lambda$. The set of singularities of $M(\lambda)$ (as an analytic function) coincides with the set $\Lambda_{0}:=\left\{\lambda: \lambda\ge 0\right\}\cup\Lambda$. Moreover, as $\lvert\rho\rvert \rightarrow \infty$ in $ \Omega,$ there holds the asymptotics
\begin{equation}\label{yj217}
 M(\lambda) = (A+i\rho A^{\perp})\left[I_{n}+\frac{h}{i\rho}-\frac{2\kappa(\rho)}{i\rho}+O\left(\frac{1}{\rho^2}\right)\right]\left(i\rho A-A^\perp \right)^{-1},
\end{equation}
where $\kappa(\rho)$ is defined in \eqref{k40}.
\end{proposition}
\begin{proof}
We only prove \eqref{yj217}. Using \eqref{jost1} and \eqref{jost2}, and noting
\begin{equation*}
(A+i\rho A^\perp )^{-1}A=A,\quad (A+i\rho A^\perp )^{-1}A^\perp =\frac{A^\perp}{i\rho},\quad A+A^\perp=I_n,
\end{equation*}
we have
\begin{equation}\label{moa}
 Ae (0,\rho)+A^{\perp}e' (0,\rho)=(A+i\rho A^\perp)\left[I_n-\frac{\omega(0,0)}{i\rho}-\frac{\kappa(\rho)}{i\rho}+O\left(\frac{1}{\rho^2}\right)\right].
\end{equation}
Substituting  \eqref{k39} and \eqref{moa} into \eqref{k2.12}, we arrive at \eqref{yj217}. The proof is complete.
\end{proof}

Let $\varphi(x,\lambda), S(x,\lambda) $ be the matrix-valued solutions of \eqref{k1} with the initial conditions, respectively,
\begin{equation}\label{yj}
    \varphi(0,\lambda)=A,\quad \varphi'(0,\lambda)=A^\perp +h,
\end{equation}
\begin{equation}\label{yj22}
  S(0,\lambda)=-A^\perp,\quad S'(0,\lambda)=A.
\end{equation}
It is easy to verify that
\begin{equation}\label{xc6}
    T(S)=I_n,\quad T(\varphi)=0_n.
\end{equation}
It is known that $\varphi(x,\lambda)$ satisfies the integral equation
\begin{equation}\label{moa2}
\varphi(x,\lambda)=A\cos \rho x+(A^\perp +h)\frac{\sin \rho x}{k}+\int_0^x\frac{\sin \rho (x-t)}{\rho }Q(t)\varphi(t,\lambda)dt.
\end{equation}
Using the method of successive approximation for \eqref{moa2}, we get that for $|\rho|\to\infty $ in $\mathbb{C}$ there hold (see, e.g., \cite{X0})
\begin{equation}\label{k23}
    \varphi(x,\lambda)=A\cos\rho x+(A^\perp +h)\frac{ \sin\rho x}{\rho}+\frac{\sin\rho x}{2\rho}\int_0^xQ(t)dtA+o\left(\frac{\exp(\lvert {\rm Im}\rho\rvert x) }{\rho}\right),
\end{equation}
\begin{equation}\label{k24}
    \varphi'(x,\lambda)=-A\rho \sin\rho x+(A^\perp +h)\cos\rho x+\frac{\cos\rho x}{2}\int_0^xQ(t)dtA+o\left({\exp(\lvert {\rm Im}\rho\rvert x)} \right).
\end{equation}
Since $S(x,\lambda)$ and $\varphi (x,\lambda)$ form a fundamental system, we have
\begin{equation}
 \phi(x,\lambda)=S(x,\lambda)C_{1}+\varphi(x,\lambda)C_{2},
\end{equation}
 where $C_{1}$, $ C_{2}$ are $n\times n$ matrices independent of $x$. Using \eqref{xc6} and \eqref{xc7}, we have $C_{1}=I_{n}$.
Using \eqref{k2.12}, we get $C_{2}=M(\lambda)$.
Thus,
\begin{equation}\label{k15}
\phi(x,\lambda)=S(x,\lambda)+\varphi(x,\lambda)M(\lambda).
\end{equation}

Denote
\begin{equation}\label{k5}
 \ell^{*} Z:= -Z''+ZQ(x),\quad     \left\langle Z,Y \right\rangle:=Z'Y-ZY',
  \end{equation}
  \begin{equation}\label{k6}
    T^*(Z):=[Z'(0)-Z(0)h]A- Z(0)A^\perp = 0_{n},
  \end{equation}
 where $Z$ is a  matrix-valued function or a row vector-valued function.
Obviously, if $Y(x,\lambda)$ and $Z(x,\mu)$ satisfy the equations$\ell Y(x,\lambda)=\lambda Y(x,\lambda)$ and
 $\ell^{*}Z(x,\mu)=\mu Z(x,\mu)$,  respectively,
 then
\begin{equation}\label{k8}
\frac{d}{dx}\left\langle Z(x,\mu),Y(x,\lambda)\right\rangle=(\lambda-\mu)Z(x,\mu)Y(x,\lambda).
\end{equation}
It is easy to show that the equation $\ell^{*}Z=\lambda Z$ has the Jost solution
satisfying  the   integral equation
\begin{equation*}
  e^*(x,\rho)=e^{i\rho x}I_n+\int_x^\infty\frac{\sin \rho(t-x)}{\rho}e^*(t,\rho)Q(t)dt,\quad x\ge0,  \quad \rho\in \Omega .
\end{equation*}
Let $\varphi^{*}(x,\lambda)$ and $S^{*}(x,\lambda)$ be the matrix-valued solutions satisfying the equation $\ell^{*}Z=\lambda Z$ and the initial conditions, respectively,
\begin{equation}\label{k9}
   \varphi^{*}(0,\lambda)=A,\quad \varphi^{*'}(0,\lambda)=A^{\perp}+h,\quad S^{*}(0,\lambda)=-A^{\perp},\quad S^{*'}(0,\lambda)=A.
   \end{equation}
Denote
\begin{equation}
    \phi^{*}(x,\lambda):=(T^{*}(e^{*}(x,\lambda)))^{-1}e^{*}(x,\rho),\quad M^{*}(\lambda):=\phi^*(0,\lambda)A+\phi^{*'} (0,\lambda)A^{\perp}.
\end{equation}
Similar to \eqref{k42} and \eqref{k15}, we also have
\begin{equation}\label{k2.38}
   \phi^{*(\nu)}(x,\lambda)=(i\rho)^{\nu}e^{i\rho x}(i\rho A+A^\perp )^{-1}(I_{n}+O(1/\rho)), \quad \lvert \rho \rvert \rightarrow \infty,\quad\rho\in \Omega.
\end{equation}
\begin{equation}\label{pp15}
\phi^{*}(x,\lambda)=S^{*}(x,\lambda)+M^{*}(\lambda)\varphi^{*}(x,\lambda),
\end{equation}
Considering \eqref{k8}, $\left\langle \phi^{*}(x,\lambda),\phi(x,\lambda)\right\rangle\ $ does not depend on $x$. Using \eqref{k15} and \eqref{pp15}, we calculate
\begin{align}\label{k16}
 \notag \left\langle\phi^{*}(x,\lambda),\phi(x,\lambda)\right\rangle\mid_{x=0}&=\phi^{*'}(0,\lambda)\phi(0,\lambda)-\phi^{*}(0,\lambda)\phi'(0,\lambda)\\
 \notag   &=(S^{*'}(0,\lambda)\varphi(0,\lambda)-S^{*}(0,\lambda)\varphi'(0,\lambda))M(\lambda)-M^{*}(\lambda)(\varphi^{*}(0,\lambda)S'(0,\lambda)\\
 \notag   &-\varphi^{*'}(0,\lambda)S(0,\lambda))+M^{*}(\lambda)(\varphi^{*'}(0,\lambda)\varphi(0,\lambda)-\varphi^{*}(0,\lambda)\varphi'(0,\lambda))M(\lambda)\\
 &=M(\lambda)-M^{*}(\lambda).
\end{align}
Due to \eqref{k8} again, we get
\begin{align}\label{k2.35}
\notag \lim\limits_{x\to\infty}\left\langle \phi^{*}(x,\lambda),\phi(x,\lambda)\right\rangle\ &=\lim\limits_{x\to\infty}({\phi^{*'}}(x,\lambda)\phi(x,\lambda)- \phi^{*}(x,\lambda)\phi'(x,\lambda) )\\
\notag &=\lim\limits_{x\to\infty}(T^{*}(e^{*}(x,\lambda)))^{-1}(e^{*'}(x,\lambda)e(x,\lambda)- e^{*}(x,\lambda)e'(x,\lambda) )(T(e(x,\lambda)))^{-1}\\&
  =0.
\end{align}
Using \eqref{k16} and \eqref{k2.35}, we have
\begin{equation}\label{k17}
 M(\lambda)=M^{*}(\lambda).
\end{equation}


\section{The uniqueness theorem}

In this section, we prove the uniqueness theorem for the solution of the following inverse problem.
\begin{ip}\label{ip1}
Given the Weyl matrix $M(\lambda)$, find $Q(x)$, $A$ and $h$.
\end{ip}
Together with $L={L}(Q,A,h)$, we consider a boundary value problem $\widetilde{L}=\textit{L}(\widetilde{Q},\widetilde{A},\widetilde{h})$ of the same form but with different $\widetilde{ \textit{Q}}$ ,$\widetilde{\textit{h}}$ and $\widetilde{\textit{A}}$. We agree that if a symbol $\alpha$ denotes an object related to $\textit{L}(Q,A,h)$, then $\widetilde{\alpha}$ will denote the analogous object related to $\textit{L}(\widetilde{Q},\widetilde{A},\widetilde{h})$.

\begin{lemma}\label{lem5}
 Define the block-matrix
\begin{equation}\label{k46}
   P(x,\lambda)=\begin{bmatrix}
 \varphi(x,\lambda) & \phi(x,\lambda)\\
 \varphi'(x,\lambda) & \phi'(x,\lambda)\\
 \end{bmatrix}\begin{bmatrix}
 \widetilde{\varphi}(x,\lambda) & \widetilde{\phi}(x,\lambda)\\
     \widetilde{\varphi'}(x,\lambda) & \widetilde{\phi'}(x,\lambda)\\
 \end{bmatrix}^{-1}.
\end{equation} If $\widetilde{A}=A$, then for all $x\ge0$ there hold the estimates
  \begin{equation}\label{px1}
    P_{jk}(x,\lambda)=\delta_{jk}I_{n}+O(1/\rho),\quad P_{21}(x,\lambda)=O(1), \quad  \lvert\rho\rvert\to\infty,\quad \rho\in\Omega, \quad 1\le j\le k\le 2,
  \end{equation}
  where $\delta_{jk}=1$ if $j=k$ and $\delta_{jk}=0$ if $j\ne k$.
  \end{lemma}
\begin{proof}
Using \eqref{k15}, \eqref{pp15}, \eqref{k8} and \eqref{k17}, together with the initial conditions \eqref{yj}, \eqref{yj22} and \eqref{k9}, we get
\begin{equation}\label{k18}
    \varphi^{*}(x,\lambda)\phi'(x,\lambda)-\varphi^{*'}(x,\lambda)\phi(x,\lambda)=I_{n},
\end{equation}
\begin{equation}\label{k19}
   \varphi^{*}(x,\lambda)\varphi'(x,\lambda)-\varphi^{*}(x,\lambda)\varphi(x,\lambda)=0_n,
\end{equation}
\begin{equation}\label{k20}
  \phi^{*'}(x,\lambda)\varphi(x,\lambda)-\phi^{*}(x,\lambda)\varphi'(x,\lambda)=I_{n},
\end{equation}
\begin{equation}\label{k21}
  \phi^{*}(x,\lambda)\phi'(x,\lambda)-\phi^{*'}(x,\lambda)\phi(x,\lambda)=0_n.
\end{equation}
Hence
\begin{equation}\label{k22}
 \begin{bmatrix}
 \varphi(x,\lambda) & \phi(x,\lambda)\\
 \varphi'(x,\lambda) & \phi'(x,\lambda)\\
 \end{bmatrix}^{-1}=\begin{bmatrix}
 \phi^{*'}(x,\lambda) & -\phi^{*}(x,\lambda)\\
 -\varphi^{*'}(x,\lambda) & \varphi^{*}(x,\lambda)\\
\end{bmatrix}.
\end{equation}
It follows from \eqref{k46} and \eqref{k22} with tilde that
\begin{equation}\label{k47}
   P(x,\lambda)=\begin{bmatrix}
 \varphi(x,\lambda) & \phi(x,\lambda)\\
 \varphi'(x,\lambda) & \phi'(x,\lambda)\\
 \end{bmatrix}\begin{bmatrix}
 \widetilde{\phi}^{*'}(x,\lambda) & -\widetilde{\phi}^{*}(x,\lambda)\\
     -\widetilde{\varphi}^{*'}(x,\lambda) & \widetilde{\varphi}^{*}(x,\lambda)\\
 \end{bmatrix}.
\end{equation}
Hence
\begin{equation}\label{k48}
     P_{j1}(x,\lambda)=\varphi^{(j-1)}(x,\lambda)\widetilde{\phi}^{*'}(x,\lambda)-\phi^{(j-1)}(x,\lambda)\widetilde{\varphi}^{*'}(x,\lambda),\quad j=1,2,
\end{equation}
\begin{equation}\label{k49}
    P_{j2}(x,\lambda)=\phi^{(j-1)}(x,\lambda)\widetilde{\varphi}^{*}(x,\lambda)-\varphi^{(j-1)}(x,\lambda)\widetilde{\phi}^{*}(x,\lambda),\quad j=1,2.
\end{equation}
Note that \eqref{k22} with tilde also implies
 \begin{equation}
\widetilde{\varphi}(x,\lambda)\widetilde{\phi}^{*'}(x,\lambda)- \widetilde{\phi}(x,\lambda)\widetilde{\varphi}^{*'}(x,\lambda)=I_{n},
 \end{equation}
  \begin{equation}
\widetilde{\varphi}'(x,\lambda)\widetilde{\phi}^{*}(x,\lambda)- \widetilde{\phi}'(x,\lambda)\widetilde{\varphi}^{*}(x,\lambda)=-I_{n},
 \end{equation}
  \begin{equation}
  \widetilde{\varphi}^{(\nu)}(x,\lambda)\widetilde{\phi}^{*(\nu)}(x,\lambda)-\widetilde{\phi}^{(\nu)}(x,\lambda)\widetilde{\varphi}^{*(\nu)}(x,\lambda)=0_{n},\quad \nu=0,1.
 \end{equation}
 Combining with \eqref{k48} and \eqref{k49}, we have
 \begin{equation}\label{px7}
     P_{11}(x,\lambda)=I_{n}+(\varphi(x,\lambda)-\widetilde{\varphi}(x,\lambda))\widetilde{\phi}^{*'}(x,\lambda)-(\phi(x,\lambda) -\widetilde{\phi}(x,\lambda))\widetilde{\varphi}^{*'}(x,\lambda),
 \end{equation}
  \begin{equation}\label{px7s}
     P_{22}(x,\lambda)=I_{n}+(\phi'(x,\lambda)-\widetilde{\phi}'(x,\lambda))\widetilde{\varphi}^{*}(x,\lambda)-(\varphi'(x,\lambda) -\widetilde{\varphi}'(x,\lambda))\widetilde{\phi}^{*}(x,\lambda),
 \end{equation}
  \begin{equation}\label{px8}
     P_{12}(x,\lambda)=(\phi(x,\lambda)-\widetilde{\phi}(x,\lambda))\widetilde{\varphi}^{*}(x,\lambda)-(\varphi(x,\lambda) -\widetilde{\varphi}(x,\lambda))\widetilde{\phi}^{*}(x,\lambda).
 \end{equation}
   \begin{equation}\label{px8s}
     P_{21}(x,\lambda)=(\varphi'(x,\lambda)-\widetilde{\varphi}'(x,\lambda))\widetilde{\phi}^{*'}(x,\lambda)-(\phi'(x,\lambda) -\widetilde{\phi}'(x,\lambda))\widetilde{\varphi}^{*'}(x,\lambda).
 \end{equation}
 Due to \eqref{k42},  \eqref{k23} and \eqref{k2.38}, as well as $A=\widetilde{A}$, and noting that $(\rho A\pm A^\perp)^{-1}A=A/\rho $, we get that  for $ \lvert\rho\rvert\to\infty, \ \rho\in\Omega$, $\nu=0,1$,
\begin{equation}\label{px5}
 (\varphi^{(\nu)}(x,\lambda)-\widetilde{\varphi}^{(\nu)}(x,\lambda))\widetilde{\phi}^{*'}(x,\lambda)=O\left(\frac{1}{\rho^{1-\nu}}\right),\quad (\phi^{(\nu)}(x,\lambda) -\widetilde{\phi}^{(\nu)}(x,\lambda))\varphi^{*'}(x,\lambda)=O\left(\frac{1}{\rho^{1-\nu}}\right),
\end{equation}
\begin{equation}\label{px6}
    (\phi(x,\lambda)-\widetilde{\phi}(x,\lambda))\widetilde{\varphi}^{*}(x,\lambda)=O\left(\frac{1}{\rho}\right),\quad (\varphi(x,\lambda) -\widetilde{\varphi}(x,\lambda))\phi^{*}(x,\lambda)=O\left(\frac{1}{\rho}\right).
\end{equation}
Substitute \eqref{px5} and \eqref{px6} into \eqref{px7}-\eqref{px8s}, respectively, we get \eqref{px1}.
  The proof of Lemma \ref{lem5} is complete.
\end{proof}

\begin{theorem}\label{th1}
If $M(\lambda)=\widetilde{M}(\lambda),$ then $Q=\widetilde{Q}$, $h=\widetilde{h}$ and $A=\widetilde{A}$.
\end{theorem}
\begin{proof}
Let us first show that $A$ is uniquely determined from $M(\lambda)$. Indeed, using the first equation in \eqref{yj111} and \eqref{yj217}, we have
\begin{align*}
   M(\lambda)=  &N\left[
                   \begin{array}{cc}
                     I_{r} & 0 \\
                     0 & i\rho I_{n-r}  \\
                   \end{array}
                 \right] N^\dagger N\left(I_{n}+N^\dagger O(1/\rho)N\right)N^\dagger N \left[
                   \begin{array}{cc}
                     \frac{1}{i\rho}I_{r} & 0 \\
                     0 & -I_{n-r}  \\
                   \end{array}
                 \right]  N^\dagger \\
                 =&N\left[
                   \begin{array}{cc}
                     I_{r} & 0 \\
                     0 & i\rho I_{n-r}  \\
                   \end{array}
                 \right] \left(I_{n}+N^\dagger O(1/\rho)N\right)\left[
                   \begin{array}{cc}
                     \frac{1}{i\rho}I_{r} & 0 \\
                     0 & -I_{n-r}  \\
                   \end{array}
                 \right]  N^\dagger\\
                 =&N\left[
                   \begin{array}{cc}
                     \frac{1}{i\rho}I_{r} & 0 \\
                     0 & -i\rho I_{n-r}  \\
                   \end{array}
                 \right]  N^\dagger+ O(1)=-i\rho A^\perp +O(1),\quad |\rho|\to\infty,\quad \rho\in \Omega,
\end{align*}
which implies
\begin{equation}\label{px2}
 A=I_n- \lim_{|\rho|\to\infty,\rho\in\Omega } \frac{M(\lambda)}{-i\rho}.
\end{equation}

Due to \eqref{k48}, \eqref{k15}, \eqref{pp15} and \eqref{k17}, we have
\begin{equation}\label{k53}
\begin{aligned}
 P_{11}(x,\lambda)= \varphi(x,\lambda)\widetilde{S}^{*'}(x,\lambda)-S(x,\lambda)&\widetilde{\varphi}^{*'}(x,\lambda)\\&
 +\varphi(x,\lambda)(\widetilde{M}(\lambda)-M(\lambda))\varphi^{*'}(x,\lambda),  \end{aligned}
 \end{equation}
\begin{equation}\label{k54}
\begin{aligned}
  P_{12}(x,\lambda)=S(x,\lambda)\widetilde{\varphi}^{*'}(x,\lambda)-\varphi(x,\lambda)&\widetilde{S}^{*}(x,\lambda)\\&
 +\varphi(x,\lambda)(M(\lambda)-\widetilde{M}(\lambda))\varphi^{*}(x,\lambda).
  \end{aligned}
 \end{equation}
Since $M= \widetilde{M}$, we have that $P_{11}(x,\lambda)$ and$P_{12}(x,\lambda)$ are both entire functions of $\lambda$ of order $\le \frac{1}{2}$. Using the Phragm\'{e}n-Lindel\"{o}f theorem (or the maximum modulus principle) and the Liouville theorem, together with Lemma \ref{lem5},  we conclude
\begin{equation}\label{k56}
 P_{11}(x,\lambda)\equiv I_n, \quad P_{12}(x,\lambda)\equiv 0_{n},
\end{equation}
which implies from \eqref{k47} that $\varphi(x,\lambda)=\widetilde{\varphi}(x,\lambda)$, and so $h=\widetilde{h}$ and  $Q(x)=\widetilde{Q}(x)$. The proof of Theorem \ref{th1} is complete.
\end{proof}
\section{Solution of the inverse problem}
In this section, we establish the main equation and show its solvability, which yields a theoretical algorithm for
Inverse Problem \ref{ip1}.
We choose an arbitrary model boundary value problem $\widetilde{L}=L(\widetilde{Q},A,\widetilde{h})$, where $A$ is obtained from \eqref{px2}. For example, we can take $\widetilde{Q}=0_n$ and $\widetilde{h}=0_n$.
Denote
\begin{equation}\label{yj41}
    D(x,\lambda,\mu):=\frac{\left\langle\varphi^{*}(x,\mu),\varphi(x,\lambda)\right\rangle}{\lambda-\mu}=\int_0^x \varphi^{*}(t,\mu)\varphi(t,\lambda)dt,
    \end{equation}
    \begin{equation}\label{yj41s}
     \widetilde{D}(x,\lambda,\mu):=\frac{\left\langle\widetilde{\varphi}^{*}(x,\mu),\widetilde{\varphi}(x,\lambda)\right\rangle}{\lambda-\mu}=\int_0^x \widetilde{\varphi}^{*}(t,\mu)\widetilde{\varphi}(t,\lambda)dt,
    \end{equation}
\begin{equation}\label{f3}
  r(x,\lambda,\mu):=\hat{M}(\mu)D(x,\lambda,\mu),\quad \widetilde{r}(x,\lambda,\mu):=\hat{M}(\mu)\widetilde{D}(x,\lambda,\mu),\quad \hat{M} =M-\widetilde{M}.
\end{equation}
Let $\mu=\tau^2$. Using \eqref{yj41s} and the similar estimate to \eqref{k23} for $\tilde{\varphi}^{*}(x,\lambda)$, we have that if $\widetilde{Q}=0_n$ and $\widetilde{h}=0_n$ then
\begin{align*}
 \notag \widetilde{D}(x,\lambda,\mu)=\int_0^x \widetilde{\varphi}^{*}(t,\mu)\widetilde{\varphi}(t,\lambda)dt
=\frac{A\tau\rho-A^{\perp}}{2\tau\rho(\rho+\tau)} \sin(\rho+\tau )x+ \frac{A\tau\rho+A^{\perp}}{2\tau\rho(\rho-\tau)} \sin(\rho-\tau )x.
\end{align*}
For the general $\widetilde{Q}$ and $\widetilde{h}$, the calculation is a little complicated. But it is easy to get that
\begin{align}\label{pp43}
\notag \widetilde{D}(x,\lambda,\mu)=&A \left[ O\left(\frac{e^{(|{\rm Im}\tau |+|{\rm Im}\rho |)x}}{|\rho\pm \tau|}\right)\cdot A+O\left(\frac{e^{(|{\rm Im}\tau |+|{\rm Im}\rho |)x}}{|\rho(\rho\pm \tau)|}\right)\cdot A^\perp\right]\\
&+A^\perp\left[ O\left(\frac{e^{(|{\rm Im}\tau |+|{\rm Im}\rho |)x}}{|\tau (\rho\pm\tau)|}\right)\cdot A+O\left(\frac{e^{(|{\rm Im}\tau |+|{\rm Im}\rho |)x}}{|\rho\tau (\rho\pm\tau)|}\right)\cdot A^\perp\right]
\end{align}
for each fixed $\rho\in \Omega$ and $|\tau|\to \infty$ in $\Omega$ or for each fixed $\tau\in \Omega$ and $|\rho|\to \infty$ in $\Omega$.
The estimate \eqref{yj217} implies that for $\lvert\tau\rvert\to\infty $ in $\Omega$,
\begin{equation}\label{pp44}
  \hat{M}(\mu)=M(\mu)-\widetilde{M}(\mu)=(A+i\tau A^{\perp})O(1/\tau)\left(i\tau A-A^\perp \right)^{-1}.
\end{equation}
Substituting \eqref{pp43} and \eqref{pp44} into \eqref{f3}, and noting \eqref{xj3}, we get
\begin{align}\label{px18}
\notag\widetilde{r}(x,\lambda,\mu)=(A+i\tau A^{\perp})&O(1/\tau)\left\{A \left[ O\left(\frac{e^{(|{\rm Im}\tau |+|{\rm Im}\rho |)x}}{|\tau(\rho\pm \tau)|}\right)\cdot A+O\left(\frac{e^{(|{\rm Im}\tau |+|{\rm Im}\rho |)x}}{|\rho\tau(\rho\pm \tau)|}\right)\cdot A^\perp\right]\right.\\
&\left.+A^\perp\left[ O\left(\frac{e^{(|{\rm Im}\tau |+|{\rm Im}\rho |)x}}{|\tau (\rho\pm\tau)|}\right)\cdot A+O\left(\frac{e^{(|{\rm Im}\tau |+|{\rm Im}\rho |)x}}{|\rho\tau (\rho\pm\tau)|}\right)\cdot A^\perp\right]\right\}.
\end{align}
Using the fact that $A(A+i\tau A^{\perp})=A$, $\tau^{-1} A^\perp(A+i\tau A^{\perp})=iA^{\perp}$, we get from \eqref{px18} that
for each $\rho\in\Omega$,
\begin{equation}\label{yj44}
 \left[O\left(1\right)A+O\left(1/\tau\right) A^{\perp}\right] \widetilde{r}(x,\lambda,\mu)=O\left(\frac{e^{(|{\rm Im}\tau |+|{\rm Im}\rho |)x}}{\lvert\tau\rvert^{2}|\rho\pm \tau|}\right),\quad \lvert\tau\rvert\to\infty ,\quad \tau\in \Omega.
    \end{equation}
The estimate of $r(x,\lambda,\mu)$ similar to \eqref{yj44} is also valid.

In the $\lambda$-plane we consider the contour $\gamma=\gamma'\cup\gamma''$ with counterclockwise circuit   (see Figure 4.1), where $\gamma'$ is a bounded closed contour encircling the set $\Lambda\cup\widetilde{\Lambda}\cup\{0\}$, and $\gamma''$ is the two-sided cut along the arc$\{\lambda:\lambda>0,\lambda\notin {\rm int}\gamma'\}$. \begin{figure}[!hbt]\center\includegraphics[scale=0.4,angle=0]{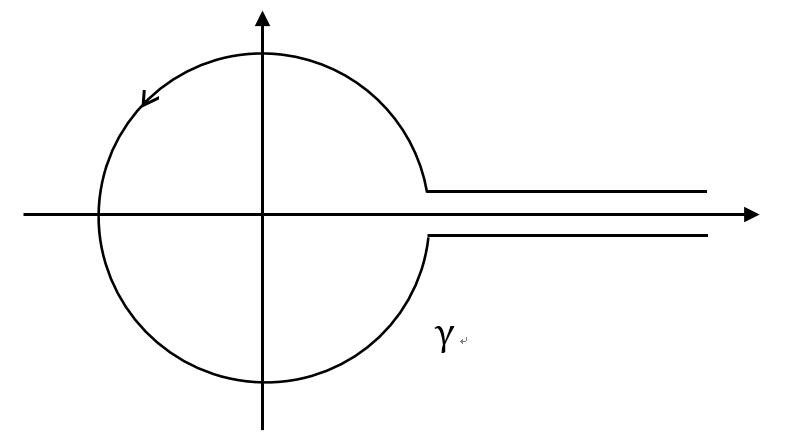} \caption{The contour $\gamma$}\end{figure}
\begin{theorem}\label{th2}
    The following relation holds
    \begin{equation}\label{yj43}
     \widetilde{\varphi} (x,\lambda)=\varphi(x,\lambda)+\frac{1}{2\pi i}\int_\gamma\varphi(x,\mu)\widetilde{r}(x,\lambda,\mu)d\mu .
    \end{equation}
    The equation \eqref{yj43} is called the main equation of Inverse Problem \ref{ip1}.
\end{theorem}
\begin{proof}
    It follows from \eqref{k23} and \eqref{yj44} that the integral in \eqref{yj43} converges absolutely and uniformly on $\gamma$, for each fixed $x\ge 0$. Denote $J_{\gamma}=\left\{\lambda:\lambda\notin \gamma\cup {\rm int}\gamma'\right \}$,$\gamma_{R}=\gamma\cap\{\lambda:\lvert\lambda\rvert\le R\}$, and $\gamma_{R}^{0}=\gamma_{R}\cup\{\lambda:\lvert\lambda\rvert= R\}$ (see Figure 4.2). \begin{figure}[!hbt]\center\includegraphics[scale=0.35,angle=0]{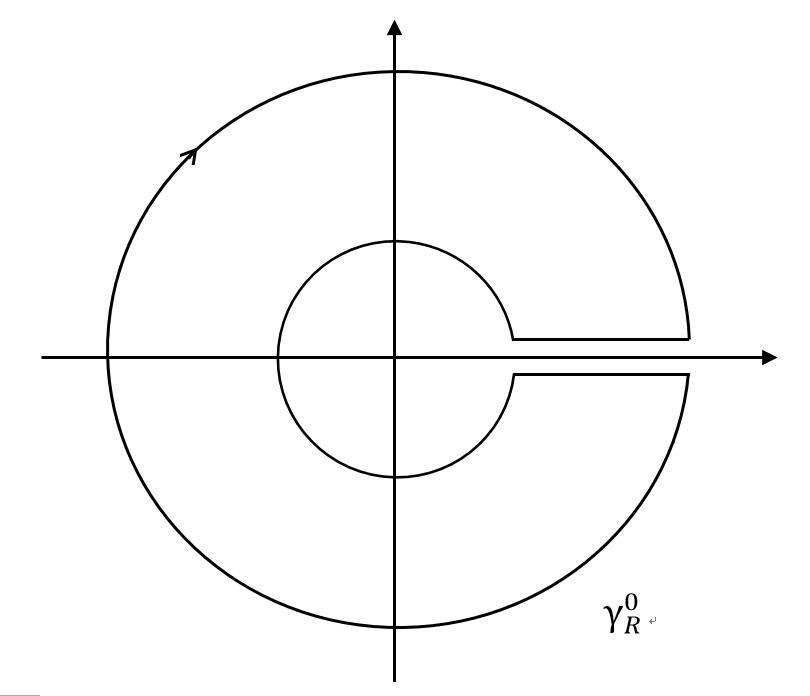}\caption{The contour $\gamma_{R}^{0}$}\end{figure}
    Using  Cauchy's integral formula, we have
    \begin{equation}\label{px22}
        P_{1k}(x,\lambda)-\delta_{1k}I_{n}=\frac{1}{2\pi i}\int_{\gamma_{R}^{0}}\frac{P_{1k}(x,\mu)-\delta_{1k}I_{n}}{\lambda-\mu}d\mu,\quad \lambda\in {\rm int} \gamma_{R}^{0},\quad k=1,2.
    \end{equation}
    Using Lemma \ref{lem5}, we get
    \begin{equation}\label{px21}
    \lim_{R\to \infty}\int_{\lvert \mu \rvert=R} \frac{P_{1k}(x,\mu)-\delta_{1k}I_{n}}{\lambda-\mu}d\mu = 0_n,\quad k=1,2.
    \end{equation}
  By  Cauchy's theorem, we have
 \begin{equation}\label{px20}
     \frac{1}{2\pi i}\int_\gamma\frac{d\mu}{\lambda-\mu}=0,\quad \lambda\in J_{\gamma}.
 \end{equation}
Using \eqref{px21} and \eqref{px20} in \eqref{px22}, we have
 \begin{equation}\label{yj417}
      P_{1k}(x,\lambda)=\delta_{1k}I_{n}+\frac{1}{2\pi i}\int_\gamma\frac{P_{1k}(x,\mu)}{\lambda-\mu}d\mu,\quad \lambda\in J_{\gamma},\quad k=1,2,
 \end{equation}
where the integral is understood in the sense of principal value.
Using \eqref{k46}, we have
    \begin{equation}\label{px23}
       \varphi(x,\lambda)=P_{11}(x,\lambda)\widetilde{\varphi}(x,\lambda)+P_{12}(x,\lambda)\widetilde{\varphi}'(x,\lambda).
    \end{equation}
    Substituting \eqref{px22} into \eqref{px23}, we get
    \begin{align}
     \notag   \varphi(x,\lambda)=&\left(I_{n}+\frac{1}{2\pi i}\int_\gamma\frac{P_{11}(x,\lambda)}{\lambda-\mu }d\mu\right)\widetilde{\varphi}(x,\lambda)+\left(\frac{1}{2\pi i}\int_\gamma\frac{P_{12}(x,\lambda)}{\lambda-\mu }d\mu\right)\widetilde{\varphi}'(x,\lambda)\\
       =&\widetilde{\varphi}(x,\lambda)+\frac{1}{2\pi i}\int_\gamma\frac{P_{11}(x,\mu)\widetilde{\varphi}(x,\lambda)+P_{12}(x,\mu)\widetilde{\varphi}'(x,\lambda)}{\lambda-\mu }d\mu.
    \end{align}
 Taking \eqref{k48} and \eqref{k49} into account, we get
 \begin{equation}\label{f15}
\begin{aligned}
\varphi(x,\lambda)=\widetilde{\varphi}(x,\lambda)+\frac{1}{2\pi i}&\int_\gamma[(\varphi(x,\mu)\widetilde{\phi}^{*'}(x,\mu)-\phi(x,\mu)\widetilde{\varphi}^{*'}(x,\mu))\widetilde{\varphi}(x,\lambda)
\\&+(\phi(x,\mu)\widetilde{\varphi}^{*}(x,\mu)-\varphi(x,\mu)\widetilde{\phi}^{*}(x,\mu))\widetilde{\varphi}'(x,\lambda)]\frac{d\mu}{\lambda-\mu}.
\end{aligned}
 \end{equation}
 Using \eqref{pp15}, \eqref{k17}, \eqref{k15}, the integral term in the right-hand side of \eqref{f15} equals
 \begin{align}\label{f16}
  \notag &\frac{1}{2\pi i}\int_{\gamma}[\varphi(x,\mu)\widetilde{M}(\mu )\widetilde{\varphi}^{*'}(x,\mu)-\varphi(x,\mu )M(\mu)\widetilde{\varphi}^{*'}(x,\mu)]\widetilde{\varphi}(x,\lambda)\frac{d\mu}{\lambda-\mu}\\
 \notag  &+\frac{1}{2\pi i}\int_{\gamma}[\varphi(x,\mu)M(\mu )\widetilde{\varphi}^{*}(x,\mu)-\varphi(x,\mu )\widetilde{M}(\mu)\widetilde{\varphi}^{*}(x,\mu)]\widetilde{\varphi}'(x,\lambda)\frac{d\mu}{\lambda-\mu}\\
   =&\frac{1}{2\pi i}\int_{\gamma}\varphi(x,\mu) (\widetilde{M}(\mu)-M(\mu))\frac{\left\langle\widetilde{\varphi}^{*}(x,\mu),\widetilde{\varphi}(x,\lambda)\right\rangle}{\lambda-\mu}d\mu.
 \end{align}
Here the terms with $S(x,\lambda)$ and $S^{*}(x,\lambda)$ vanish, since they are entire functions of $\lambda$ and the Cauchy's theorem is used.
 Substituting\eqref{f16} with \eqref{f3} and \eqref{yj41s} into \eqref{f15},  we get \eqref{yj43}.
The proof is complete.
\end{proof}

\begin{lemma}\label{th3}
 The following relations hold
 \begin{equation}\label{yj421}
     \widetilde{r}(x,\lambda,\mu)-r(x,\lambda,\mu)-\frac{1}{2\pi i}\int_\gamma  r(x,\xi,\mu)\widetilde{r}(x,\lambda,\xi)d\xi=0_n,
 \end{equation}
  \begin{equation}\label{px3}
     \widetilde{r}(x,\lambda,\mu)-r(x,\lambda,\mu)-\frac{1}{2\pi i}\int_\gamma  \widetilde{r}(x,\xi,\mu){r}(x,\lambda,\xi)d\xi=0_n.
 \end{equation}
\end{lemma}
\begin{proof}
The proof of \eqref{yj421} is the same as that in Lemma 4.2 in \cite{FY}. So, we omit it. One can also prove \eqref{px3} by changing the places of $L$ and $\tilde{L}$.
Indeed, if we define\begin{equation*} r_1(x,\lambda,\mu):=\check{M}(\mu)D(x,\lambda,\mu),\quad \widetilde{r}_1(x,\lambda,\mu):=\check{M}(\mu)\widetilde{D}(x,\lambda,\mu),\quad \check{M}=\widetilde{M}-M,\end{equation*} then we can also obtain \begin{equation*}   {r}_1(x,\lambda,\mu)-\widetilde{r}_1(x,\lambda,\mu)-\frac{1}{2\pi i}\int_\gamma  \widetilde{r}_1(x,\xi,\mu) {{r}}_1(x,\lambda,\xi)d\xi=0_n,\end{equation*} which is the same as \eqref{px3}.
\end{proof}

Let us consider the Banach space $X$ of matrix-valued functions, defined by
\begin{equation*}
  X=\{z(\lambda):z(\lambda) \text{ is continuous on $\gamma$ and satisfies that $z(\lambda)(A+i\rho A^\perp)$ is bounded on $\gamma$}  \},
\end{equation*}
with the norm
 $$\|z\|_X:=\max_{\lambda\in\gamma}\|z(\lambda)(A+i\rho A^\perp)\|.
$$
\begin{theorem}
    For each fixed $x\ge 0$, the main equation \eqref{yj43} has a unique solution $\varphi(x,\cdot) \in X$.
\end{theorem}
\begin{proof}
    Define the maps:
    \begin{equation*}
        \widetilde{D}z(\lambda)=z(\lambda)+\frac{1}{2\pi i}\int_{\gamma}z(\mu )\widetilde{r}(x,\lambda,\mu )d\mu,
    \end{equation*}
    \begin{equation*}
       Dz(\lambda)=z(\lambda)-\frac{1}{2\pi i}\int_{\gamma}z(\mu )r(x,\lambda,\mu )d\mu.
    \end{equation*}
   Using \eqref{px18}, it is easy to verify that $D$ and $\tilde{D}$ are bounded linear operators on $X$. 
    Then
        \begin{align}\label{px14}
\notag \widetilde{D}Dz(\lambda)&=\widetilde{D}\left (z(\lambda)-\frac{1}{2\pi i}\int_{\gamma}z(\mu )r(x,\lambda,\mu )d\mu\right)\\
\notag=&z(\lambda) +\frac{1}{2\pi i} \int_{\gamma} z(\mu) \widetilde{r}(x,\lambda,\mu )d\mu-\widetilde{D}\left(\frac{1}{2\pi i} \int_{\gamma}z(\mu )r(x,\lambda ,\mu )d\mu\right)\\
\notag=&z(\lambda)+\frac{1}{2\pi i}\int_{\gamma}z(\mu )[\widetilde{r}(x,\lambda,\mu )-{r}(x,\lambda,\mu )]d\mu \\
&\quad \quad -\frac{1}{2\pi i}\int _{\gamma}\left[\frac{1}{2\pi i}\int _{\gamma}z(\mu )r(x,\xi,\mu )d\mu\right]\widetilde{r}(x,\lambda,\xi )d\xi  .
        \end{align}
For each $z\in X$, using \eqref{yj44} and \eqref{pp43}, we have
    \begin{equation*}
        \int_{\gamma} \int_{\gamma}\| z(\mu)r(x,\xi ,\mu)\widetilde{r}(x,\lambda,\xi )\| d \mu d\xi  < +\infty.
    \end{equation*}
 Due to Fubini's theorem,    we have
    \begin{equation*}
    \int _{\gamma}\int _{\gamma}z(\mu )r(x,\xi,\mu )d\mu\widetilde{r}(x,\lambda,\xi )d\xi = \int _{\gamma}\int _{\gamma}z(\mu )r(x,\xi,\mu )\widetilde{r}(x,\lambda,\xi) d\xi d\mu.
    \end{equation*}
It follows from \eqref{px14} that
    \begin{equation}\label{yj423}
       \widetilde{D}Dz(\lambda)=z(\lambda) +\frac{1}{2\pi i} \int_{\gamma}z(\mu)\left[\widetilde{r}(x,\lambda,\mu)-r(x,\lambda,\mu)-\frac{1}{2\pi i}\int_{\gamma}r(x,\xi,\mu)\widetilde{r}(x,\lambda,\xi )d\xi \right]d\mu,
    \end{equation}
    which implies from \eqref{yj421} that $\widetilde{D}Dz(\lambda) =z(\lambda).$
 Similarly, using \eqref{px3}, we also have $D\widetilde{D}z(\lambda) =z(\lambda).$
Hence the operator $\widetilde{D}$ has a bounded inverse operator, and so the main equation \eqref{yj43} is uniquely solvable for each fixed $x\ge 0$.
\end{proof}

From the solvability of the main equation \eqref{yj43}, we can get a theoretical reconstruction algorithm for Inverse Problem \ref{ip1}.

\begin{algorithm}Let the matrix-function $M(\lambda)$  be given.\\
  Step 1. Obtain $A$ by \eqref{px2}.
  \\
  Step 2. Choose ${L}(\widetilde{Q},A,\widetilde{h})$ and calculate $\widetilde{\varphi}(x,\lambda)$ and $\widetilde{r}(x,\lambda,\mu)$.
\\
   Step 3.   Find $\varphi(x,\lambda)$ by solving the main equation \eqref{yj43}.
\\
   Step 4.   Get $Q(x)$  and $h$ via
       $Q(x) = \varphi''(x,\lambda)(\varphi(x,\lambda))^{-1} - \lambda I_{n}$,
   $h = \varphi'(0,\lambda) - iA^{\perp}$.
\end{algorithm}

\noindent\textbf{Data availability.}
Data sharing is not applicable to this article as no new data were created or analyzed in this study.
\\[2mm]

\noindent\textbf{\large Appendix}
\\[-2mm]

\appendix
\setcounter{equation}{0}
\renewcommand\theequation{A.\arabic{equation}}

In Appendix, we analyse how to transform the condition \eqref{xc1} with \eqref{xc2} into the form of \eqref{k2}. Since $U$ is a unitary matrix, there exists a unitary matrix $N$ such that
\begin{equation*}
N^\dagger U N ={\rm diag} \left\{-e^{-2i\theta_1},\cdot\cdot\cdot ,-e^{-2i\theta_r},-I_{n-r} \right\}.
\end{equation*}
where $\theta_j\in (0,\pi)$, $j=1,...,r$.
It follows from \eqref{xc2} that
 \begin{align}
 \notag N^\dag A_1N&={\rm diag}\left\{\frac{1-e^{-2i\theta_1}}{2},\cdot\cdot\cdot,\frac{1-e^{-2i\theta_{r}}}{2},0_{n-r}\right\},\\
 \notag N^\dag B_1N&={\rm diag}\left\{\frac{(e^{-2i\theta_1}+1)}{2i},\cdot\cdot\cdot,\frac{(e^{-2i\theta_{r}}+1)}{2i},-iI_{n-r}\right\}.
 \end{align}
Denote
\begin{equation*}
A_0:=N {\rm diag} \left\{\frac{2}{1-e^{-2i\theta_1}},\cdot\cdot\cdot ,\frac{2}{1-e^{-2i\theta_r}},iI_{n-r} \right\}N^\dagger,
\quad
A:=A_1A_0 , \quad B:= B_1A_0.
\end{equation*}
Then $A_0$ is invertible and $A,B$ have the following forms, respectively,
\begin{equation}\label{yj111}
 A =N  {\rm diag} \left\{I_{r} , 0_{n-r}\right\}N^\dagger ,\quad B=-N{\rm diag} \left\{\cot \theta_1,\cdot\cdot\cdot ,\cot \theta_r,I_{n-r} \right\}N^\dagger.
\end{equation}
Note that the form of $A$ in \eqref{yj111} is equivalent to that $A$ is an orthogonal projection matrix, i.e., $A^\dagger =A=A^2$. We can also rewrite  $B$ as the form
\begin{equation*}\label{yj112}
B=-A^\perp +h,\quad h:= -N {\rm diag} \left\{\cot \theta_1,\cdot\cdot\cdot ,\cot \theta_r , 0_{n-r}\right\} N^\dagger
\end{equation*}
where $A^\perp=I-A$, and $h$  satisfies $Ah=hA=AhA=h=h^\dagger$.
Premultiplying \eqref{xc1} by the invertible matrix $A_0^\dagger$, and get the boundary condition \eqref{k2}.

\end{document}